\documentclass[11pt, reqno]{amsart}

\evensidemargin
\oddsidemargin
\usepackage{amsmath, amsthm, latexsym, amssymb, bbm}
\usepackage[dvips]{color}

\usepackage{calc,  subfigure, bm}
\usepackage[dvips]{graphicx}

\newcommand{\eps}{\varepsilon}

\newcommand{\e}{\varepsilon}

\renewcommand{\phi}{\varphi}

\newcommand{\R}{\mathbb R}

\newcommand{\N}{\mathbb N}

\newcommand{\E}{\mathbb E}
\renewcommand{\P}{\mathbb P}

\newtheorem{rem}{Remark}
\newtheorem{theorem}{Theorem}
\newtheorem{lemma}[theorem]{Lemma}
\newcommand{\heap}[2]  {\genfrac{}{}{0pt}{}{#1}{#2}}
\newcommand{\sfrac}[2] {\mbox{$\frac{#1}{#2}$}}

\newcounter{remnr}
{\nopagebreak {\hfill{$\diamond$}}\\ }

\renewcommand{\phi}{\varphi}

\renewcommand{\P}{\mathbb{P}}
\renewcommand{\E}{\mathbb{E}}

\begin{document}


\title[A conditioning principle for Galton--Watson trees]
{\Large A conditioning principle for\\ Galton--Watson trees}

\author[Nathana\"el Berestycki, Peter M\"orters and Nadia Sidorova]{Nathana\"el Berestycki, Peter M\"orters and Nadia Sidorova}

\maketitle


\vspace{0.2cm}



\begin{quote}{\small {\bf Abstract: }
We show that an infinite Galton-Watson tree, conditioned on its martingale limit being smaller than $\eps$,
converges as $\eps\downarrow 0$ in law to the regular $\mu$-ary tree, where $\mu$ is the essential minimum of the
offspring distribution. This gives an example of entropic repulsion where the limit has no entropy.}
\end{quote}
\vspace{0.5cm}

\section{Introduction and statement of the result}

The problem of \emph{conditioning principles} can be formulated in the following way: Given that some
quantity averaged over a large number of individual random variables shows highly unlikely behaviour,
describe the conditional law of an individual sample. This situation arises frequently in statistical mechanics,
where the random variables describe individual features of particles (for example their velocity) and the ensemble
of particles is subject to some constraint (for example a fixed energy per particle). The distribution of the individual
feature given the constraint is then referred to as the \emph{micro-canonical} distribution of the system.
The most famous result in this respect is the \emph{Gibbs conditioning principle}, which loosely speaking says that
under the condition that the empirical measure
$$L_n=\frac1n\sum_{i=1}^n \delta_{X_i}$$
of a family of independent random variables $X_1, ,\ldots, X_n$ with law $P$ belongs to some set~$A$,
the law of $X_1$ converges to the probability measure~$Q$ that minimizes the relative entropy $H( Q \,| P)$
subject to the constraint $Q\in A$. 
There exist several refinements of this result
describing rigorously the precise asymptotic strategy by which the random variables realize the large
deviation event $\{L_n\in A\}$. See the book of Dembo and Zeitouni~\cite{dzbook} for more on the classical Gibbs conditioning principle,
\cite{C, DZ, SZ} for refinements, and \cite{DSZ, MV, MN} for further examples of conditioning principles.
\medskip

The conditioning principle of the present paper deals with Galton--Watson trees with a nondegenerate offspring
variable~$N$ satisfying $P(N=0)=0$ and $EN \log N<\infty$. Let $a:=EN$ be the mean offspring number.
We denote by $(Z_n \colon n=0,1,\ldots)$ the sequence of generation sizes of the Galton Watson tree and note
that by definition $Z_0=1$. By the Kesten-Stigum theorem the \emph{martingale limit} 
$$W:=\lim_{n\to\infty} \frac{Z_n}{a^n}$$
is well-defined and strictly positive almost surely. Note that $W$ can be seen as a random constant factor in front of
a deterministic exponential  growth term~$a^n$, which together determine the asymptotics of the generation size~$Z_n$.
We are interested  in the limit behaviour of $Z_1$,  or more generally of the entire tree, when we condition on the
large deviation event that the martingale limit $W$ is smaller than some $\eps\downarrow 0$.
\medskip


For the formulation of the result we denote by $\mathcal T$ the space of all rooted trees with the property that every vertex
has finite degree. A metric~$d$ on this space is uniquely determined by the requirement that $d(T_1,T_2)=e^{-n}$,
when $n$ is maximal with the property that the trees $T_1$ and $T_2$ coincide up to the $n$th generation. This makes $(\mathcal T,d)$ a
complete, separable metric~space.
\medskip

\begin{theorem} \label{T}
Suppose $N$ is a random variable on the positive integers satisfying the condition
$EN \log N<\infty$, and denote
$$\mu:=\min \big\{ n\in\N \colon \P(N=n)>0 \big\} \geq 1.$$
Suppose that $T$ is a Galton--Watson tree with offspring variable $N$ and that
$W$ is the associated martingale limit. Then, as $\eps\downarrow 0$, conditionally on the
event $\{W<\eps\}$ the tree $T$ converges in law on $(\mathcal T, d)$ to the regular $\mu$-ary tree. Equivalently,
for all $k\in\N$,
$$\lim_{\eps\downarrow 0} \P\big( Z_k=\mu^k \, \big| \, W<\eps\big) = 1,$$
where $Z_k$ denotes the size of the $k$th generation.
\end{theorem}

From the point of view of large deviations theory this result is quite surprising, at least at a first glance.
One would expect that the limiting behaviour represents the optimal strategy by which the event $W=0$ is realized  and
that this strategy depends on the details of the law of~$N$. Moreover, there seems to be no good
reason why in the limit the \emph{growth rate} of the tree should drop dramatically, or in fact why it should drop at all,
as we only require the constant to be small.  Above all, the probability of seeing a $\mu$-ary tree up to the $n$th generation
may be arbitrarily small and can certainly be much smaller than those of other trees with $Z_n\leq \eps a^n$.

\medskip This becomes even more intriguing if the result is put in the context of \emph{entropic repulsion}. This is an expression used by physicists to convey 
the idea that entropy maximisation may force certain systems to obey properties that are not obviously imposed on them \emph{a priori}. This phenomenon has been 
studied mathematically in the context of the two-dimensional harmonic crystal with hard wall repulsion by Bolthausen et al \cite{bdg}, where the following result was proved. 
Consider the discrete Gaussian free field $(\phi_x)_{x \in D_n}$ on a planar domain with mesh size $1/n$. If the field is conditioned to be nonnegative everywhere, 
then the typical value of the field $\phi_x$ at any point $x$ in the interior of the domain will be highly concentrated near the value $(4/\pi) \log n$ with 
overwhelming probability as $n \to \infty$, and in particular under this conditioning  the value of $\phi_x$ diverges to infinity. An analogous phenomenon is 
studied by Benjamini and Berestycki \cite{BB} and \cite{bb2}, where it is shown that conditioning a one-dimensional Brownian motion on some self-repelling behaviour may force 
the process to satisfy a strongly amplified version of the constraint. Usually, the reason entropic repulsion may arise is in order to increase the entropy of the 
system, i.e., make room for fluctuations. Thus the eventual state of the system is a compromise between the energy cost of adopting an unusual behaviour and the 
entropic benefits. Theorem \ref{T} may also be cast in this framework, as it shows that the effect of requiring the constant $W$ to be small is to change the 
overall exponential growth rate from $a$ to $\mu$. However, if the limiting state of the system is the regular $\mu$-ary tree, which is non-random, what could the 
entropic benefits possibly be?
\medskip

\pagebreak[3]

The resolution of this apparent paradox comes from understanding the inhomogeneity of the optimal strategy,
and can be explained by a closer look at the formula
$$Z_n \sim W \, a^n.$$
While the growth rate $\log a$ is purely asymptotic, i.e.\ depends only on the offspring numbers \emph{after} any given generation,
the growth constant $W$ depends heavily on the \emph{initial} generations of the tree. It turns out that, roughly speaking,
the collection of trees which form the optimal strategy to achieve $W<\eps$ have minimal offspring for a few generations,
the exact number depending on $\eps$,
and causes high entropic and energetic cost but only for a small number of generations,
and then switch to  growth with the natural rate~$\log a$. The initial behaviour ensures that $W$ is small
at a mimimal probabilistic cost, because for all but a finite number of generations the trees
can have their natural growth. The topology on $\mathcal T$ compares trees starting from their root
so that in the limit we only see the behaviour in the initial generations. This leads to a limiting
object with minimal growth rate at all generations and creates the illusion of a drop in the growth
rate for the optimal strategy. A somewhat similar phenomenon is observed by Bansaye and Berestycki \cite{BanBer} in the context of branching processes in random environment, although they consider situations where the growth rate is directly conditioned to be atypical.%
\smallskip

Our interest in this `paradox' does not come from the study of trees alone.
Indeed, M\"orters and Ortgiese~\cite{ortgiese} describe a range of problems, mostly related to local times of Brownian motion,
which have a similar intrinsic structure and could therefore also satisfy loosely analogous conditioning principles.
However, we shall defer the discussion of such problems to a different place as, unlike in these problems, the main mathematical difficulty
here is related to the discrete nature of the distribution~$N$. 

\section{Proof of Theorem~1}

Denote $p_k:=P(N=k)$ and recall that by our definitions
$p_0=\cdots=p_{\mu-1}=0, \, p_\mu>0.$
The basic idea of the proof is is to combine tail asymptotics at zero for the random variable~$W$ with
the self-similarity property of Galton--Watson trees, which states that, for every $n\in\N$,
\begin{equation}\label{SSP}
W=\frac1{a^n} \sum_{i=1}^{Z_n} W_{i},
\end{equation}
where $W_{i}$, $i=1,2,\ldots$ are independent variables with the same law as $W$, independent of~$Z_n$.
This follows easily from the decomposition of the tree according to the ancestry in the $n$th generation.
\smallskip%

We first give the proof in the case $\mu=1$, which is very simple as in this case the $\mu$-ary tree
is degenerated and has no exponential growth. In this case the tail at zero of the random variable~$W$
is fat, more precisely there exist constants $0<c<C$ such that
$$c\, \eps^\tau \leq \P\big( W< \eps \big) \leq C \, \eps^\tau \quad \mbox{ for all } 0<\eps<1,$$
where $\tau:=-\log p_1/\log a$, see~\cite[Theorem 1(a)]{ortgiese} for a simple proof. Using~\eqref{SSP} we infer
\begin{align*}
\P\big( Z_n>1, W<\eps\big)
& \leq P\big( W_{1} + W_{2}< a^n \eps\big) \leq P\big( W< a^n \eps\big)^2\\
& \leq C^2\, (a^n \eps)^{2\tau} \leq \P\big( W<\eps\big) \, \big( \sfrac{C^2}{c} a^{2n\tau}\big)  \eps^\tau,
\end{align*}
and hence
$$\P\big( Z_n>1 \mid  W<\eps\big) \leq \big( \sfrac{C^2}{c} a^{2n\tau}\big) \, \eps^\tau
\stackrel{\eps\downarrow 0}{\longrightarrow} 0,\\[1mm]$$
as required to complete the proof in the case~$\mu=1$.
\pagebreak[3]

Now we consider the case $\mu>1$ and assume that $p_{\mu}\neq 1$ to avoid trivialities.
We define the \emph{B\"ottcher constant} $\beta\in (0,1)$ by $$a^{\beta}=\mu.$$
A function $V\colon (0,\infty)\to (0,\infty)$ is called \emph{multiplicatively periodic} with
period $\lambda\neq 1$ if $V(\lambda x)=V(x)$ for all $x>0$.
Biggins and Bingham~\cite[Theorem 3]{biggins} show that there exists a real-analytic
multiplicatively periodic function $M\colon(0,\infty)\to (0,\infty)$ with period $a^{1-\beta}=a/\mu>1$ such
that
\begin{align}
\label{tail2}
-\log \P\big(W<x\big)=M(x) \,x^{-\frac{\beta}{1-\beta}}+o\big(x^{-\frac{\beta}{1-\beta}}\big)
\qquad \text{as }x\downarrow 0,
\end{align}
see also Fleischmann and Wachtel~\cite{FW09} for refinements of this statement.
A key argument in the proof of~\eqref{tail2} is to relate the left tail of a positive random
variable to the behaviour of its Laplace transform at infinity in a way reminiscent
of the Tauberian theorem of de Bruijn, see~\cite[Theorem~4.12.9]{regvar}. In the next lemma
we generalise this result to the case of several independent copies of $W$ using the same
basic method of proof as in~\cite{biggins}.

\pagebreak[3]

\begin{lemma}\label{l_tailm}
Let $X$ be a positive random variable such that, for some $a>1$ and $\beta\in(0,1)$, for all $s>0$,
\begin{align*}
\lim_{n\to\infty}\frac{\log \E\exp\{-sa^n X\}}{a^{n\beta}}=:k(s)
\end{align*}
with some real-analytic function $k$ on $(0,\infty)$. Then there exists
a real-analytic multiplicatively periodic function $V\colon(0,\infty)\to (0,\infty)$
with period $a^{1-\beta}$ such that, for any $m\in \N$ and $X_1, \dots, X_m$ independent
with the same distribution as $X$, we have
\begin{align*}
-\log \P\big(X_1+\cdots+X_m<x\big)=m\,V(x/m)(x/m)^{-\frac{\beta}{1-\beta}}+o\big(x^{-\frac{\beta}{1-\beta}}\big)
\qquad \text{as }x\downarrow 0.
\end{align*}
\end{lemma}

\begin{proof}
Let $Y_n$ 
be real-valued random variables and denote, for $s>0$,
$$k_n(s):=\log \E \exp\{sY_n\}\in (-\infty,\infty].$$
Assume that for some sequence of positive numbers $b_n\uparrow\infty$, we have
\begin{align*}
\lim_{n\to\infty}\frac{k_n(s)}{b_n}=:\hat k(s)\in (-\infty,\infty].
\end{align*}
The Fenchel dual of $\hat k$ is given by
$\hat k^*(x):=\sup_{s>0}\{xs-\hat k(s)\}\in (-\infty,\infty]$.
By a variant of the G\"artner--Ellis theorem, see~\cite[Corollary 1]{biggins}, we have
\begin{align*}
\lim_{n\to\infty}\frac{-\log \P(Y_n\ge b_ny)}{b_n}=\hat k^*(y)
\quad\mbox{for all $y\in (\lim\limits_{s\downarrow 0}\hat k'(s),\lim\limits_{s\uparrow\infty}\hat k'(s))$.}
\end{align*}
We apply this first to the sequence $Y_n:=-a^nX$ and observe that $\hat k=k$ in this case.
Note that $k$ satisfies, by definition,
$k(s)= v(s)s^\beta$ for a multiplicatively periodic function~$v$ with period~$a$. Using
that $k$ is strictly convex and decreasing we get $\lim_{s\downarrow 0}k'(s)=:-\delta<0$, and
$$0\geq \lim_{s\uparrow\infty}k'(s) = \lim_{n\to\infty}k'(a^{n+1}) \geq \lim_{n\to\infty} \frac{k(a^{n+1})-k(a^n)}{a^{n+1}-a^n}
= \lim_{n\to\infty} a^{n(\beta-1)} \,v(1)\,\frac{a^{\beta}-1}{a-1}=0.$$
Therefore,
\begin{align*}
\lim_{n\to\infty}\frac{-\log \P(-a^nX\ge a^{n\beta}y)}{a^{n\beta}}=k^*(y) \quad \mbox{for all $y\in (-\delta,0)$, }
\end{align*}
Setting $x=-y$ and rearranging,
\begin{align*}
\lim_{n\to\infty}\frac{-\log \P(X\le a^{-n(1-\beta)}x)}{a^{n\beta}x^{-\beta/(1-\beta)}}=
k^*(-x)x^{\frac{\beta}{1-\beta}}=:V(x), \quad\mbox{for $x\in (0,\delta)$,}
\end{align*}
where 
$V$ is real-analytic and multiplicatively periodic with period~$a^{1-\beta}$.
\smallskip

Now, consider $\tilde Y_n=-a^n(X_1+\cdots+X_m)$. We have
\begin{align*}
\lim_{n\to\infty}\frac{\log \E\exp\{-sa^n(X_1+\cdots+X_m)\}}{a^{n\beta}}
=m\lim_{n\to\infty}\frac{\log \E\exp\{-sa^n X\}}{a^{n\beta}}=mk(s),
\end{align*}
and hence
\begin{align*}
\lim_{n\to\infty}\frac{-\log \P(-a^n(X_1+\cdots+X_m)\ge a^{n\beta}y)}{a^{n\beta}}=
\sup_{s>0}\big\{ys-mk(s)\big\}=mk^*(y/m)
\end{align*}
for all $y\in (-m\delta ,0)$. Setting $x=-y$ and rearranging we obtain, for all $x\in (0,m\delta)$,
\begin{align*}
\lim_{n\to\infty}\frac{-\log \P(X_1+\cdots+X_m\le a^{-n(1-\beta)}x)}{a^{n\beta}x^{-\beta/(1-\beta)}}=
mk^*(-x/m)\,x^{\frac{\beta}{1-\beta}}=V(x/m)\,m^{\frac{1}{1-\beta}}.
\end{align*}
Denote by $H_n(x)$ the fraction on the left hand side, and by $H(x)$ the right hand
side of the display above.
Further, denote $$\hat H_n(x):=H_n(x)\,x^{-\frac{\beta}{1-\beta}} \quad \mbox{ and } \hat H(x):=H(x)\,x^{-\frac{\beta}{1-\beta}}.$$
Let $I=[a^{-2(1-\beta)}m\delta, a^{-(1-\beta)}m\delta]$.
Then $\hat H_n$ converges to $\hat H$ pointwise on $I$, $\hat H$ is continuous on $I$, and each $\hat H_n$
is decreasing. Hence $\hat H_n$ converges to $\hat H$ uniformly on $I$, see e.g.~\cite{BH}, and
therefore $H_n$ converges uniformly to $H$ on $I$.
By the periodicity of $V$ we have
\begin{align*}
\sup_{x\in a^{-n(1-\beta)}I}&\, \big|\,x^{\frac{\beta}{1-\beta}}\,\log \P\big(X_1+\cdots+X_m<x\big)
+V(x/m)\,m^{\frac{1}{1-\beta}}\big| 
=\sup_{x\in I}\big|H_n(x)-H(x)\big|,
\end{align*}
and hence
\begin{align*}
\sup_{x\le a^{-N(1-\beta)}m\delta}&\,\big|\,x^{\frac{\beta}{1-\beta}}\,\log P\big(X_1+\cdots+X_m<x\big)
+V(x/m)\,m^{\frac{1}{1-\beta}}\big| \\&
=\sup_{n> N}\sup_{x\in I}\big|H_n(x)-H(x)\big|,
\end{align*}
which converges, as $N\to \infty$, to zero as required.
\end{proof}

The next lemma states a basic property of real analytic, multiplicatively
periodic functions.

\begin{lemma}
\label{a1}
Let $V\colon(0,\infty)\to (0,\infty)$ be a real-analytic, multiplicatively
periodic function and let $\gamma>0$. Suppose that $B$ be a dense subset of
$[1,\infty)$ such that
\begin{align*}
\liminf_{\e\to 0}\big(V(\e/b)\,b^{\gamma}- V(\e)\big)\ge 0\qquad \text{ for all } b\in B.
\end{align*}
Then, for any $b_0>1$,
\begin{align*}
\inf_{\heap{\e>0}{b\ge b_0}} \big(V(\e/b)\,b^{\gamma}- V(\e)\big)>0.
\end{align*}
\end{lemma}

\begin{proof} Define a real-analytic function $f\colon\R\to \R$ by $f(y)=\log V (e^{-y})$. Then
\begin{align*}
V(\e/b)b^{\gamma}- V(\e)=\exp\{f(-\log\e+\log b)\}b^{\gamma}-\exp\{f(-\log\e)\},
\end{align*}
and, substituting $x=-\log\e$ and $\delta=\log b$, we obtain that it suffices to show that
\begin{align}
\label{nonneg}
\liminf_{x\to\infty} & \big(e^{f(x+\delta)+\gamma\delta}-e^{f(x)}\big)\ge 0
\qquad \text{ for all }\delta\in D:=\{\log b \colon b\in B\}\\
& \Longrightarrow
\inf_{\heap{x\in \R}{\delta\ge\delta_0}}\big(e^{f(x+\delta)+\gamma\delta}-e^{f(x)}\big)>0
\quad \mbox{for any $\delta_0>0$.} \label{pos}
\end{align}
By periodicity of $f$, the statement of~\eqref{nonneg} is equivalent to
$f(x+\delta)+\gamma\delta\ge f(x)$ for all $x\in\R$ and $\delta\in D$.
As $B$ is dense in $[1,\infty)$, $D$ is dense in $[0,\infty)$ and,
using the continuity of $f$, 
\begin{align}
\label{ineq}
f(x+\delta)+\gamma\delta\ge f(x) \quad\text{ for all } x\in\R \text{ and } \delta\ge 0.
\end{align}
Suppose
that~\eqref{pos} is not true and the infimum is equal to zero. Since $f$ is periodic and
$$\lim_{\delta\uparrow\infty} \inf_{x\in\R} \big\{ e^{f(x+\delta)+\gamma\delta}-e^{f(x)}\big\} = \infty,$$
the infimum in~\eqref{pos} is attained
at some point~$(\hat x,\hat \delta)$. As this infimum is zero, we infer that
\begin{align}
\label{eq}
f(\hat x+\hat \delta)+\gamma\hat \delta= f(\hat x).
\end{align}
Let $\eta\in [0,\hat\delta]$. Using~\eqref{ineq} for the points $\hat x, \eta$ and $\hat x+\eta, \hat\delta-\eta$,
we obtain
$f(\hat x+\eta)+\gamma\eta\ge f(\hat x)$ and
$f(\hat x+ \hat\delta)+\gamma(\hat\delta-\eta)\ge f(\hat x+\eta)$.
The second inequality together with~\eqref{eq} implies $f(\hat x)\ge f(\hat x+\eta)+\gamma\eta$, which
together with the first inequality gives $f(\hat x+\eta)=f(\hat x)-\gamma\eta$ for all
$\eta\in [0,\hat\delta]$. Hence $f$ is linear and non-zero on $[\hat x,\hat x+\hat\delta]$. As it is real-analytic
it must be linear on $\R$, contradicting the periodicity of $f$.
\end{proof}

We now return to the study of the martingale limit~$W$. A result from~\cite{biggins1} states that
\begin{align*}
\lim_{n\to\infty}\frac{\log \E\exp\{-s a^n W\}}{a^{n\beta}}=k(s)
\end{align*}
for some real-analytic function $k$ on $(0,\infty)$. Using Lemma~\ref{l_tailm} we infer from this that, for some
real-analytic and multiplicatively periodic $M\colon(0,\infty)\to(0,\infty)$ with period $a^{1-\beta}$ we have,
for any $m\in\N$,
\begin{equation}\label{la2}
-\log \P\big(W_1+\cdots+W_m<\eps\big)=m\,M(\eps/m)(\eps/m)^{-\frac{\beta}{1-\beta}}+o\big(\eps^{-\frac{\beta}{1-\beta}}\big)
\qquad \text{as }\eps\downarrow 0.
\end{equation}
Using first \eqref{SSP}, then \eqref{la2}, and finally $a^{\beta}=\mu$ and periodicity of~$M$, we get
\begin{align*}
\log\P(W<\e\,|\,Z_n=m)
& 
=\log\P\Big(\sum_{i=1}^mW_i<\e a^n\Big)\\
&=-M\big(\e a^n/m\big)\,m^{\frac{1}{1-\beta}}\,a^{-\frac{\beta n}{1-\beta}}
\e^{-\frac{\beta}{1-\beta}}+o\big(\e^{-\frac{\beta}{1-\beta}}\big)\\
&=-M\big(\e \mu^n/m\big)\,(m/\mu^n)^{\frac{1}{1-\beta}} \,
\e^{-\frac{\beta}{1-\beta}}+o\big(\e^{-\frac{\beta}{1-\beta}}\big).
\end{align*}
Combining with~\eqref{la2} again we obtain
\begin{equation}\label{gqrks}
\begin{aligned}
\log \P(W<\e)& -\log \P(W<\e\,|\,Z_n=m) \\
&=\left(M(\e \mu^n/m)(m/\mu^n)^{\frac{1}{1-\beta}}-M(\e)\right)\e^{-\frac{\beta}{1-\beta}}
+o(\e^{-\frac{\beta}{1-\beta}}).
\end{aligned}\end{equation}
Lemma~\ref{a1} enables us to analyse the bracketed term.

\begin{lemma}
\label{mmm}
For any $b_0>1$, we have
\begin{align*}
\inf_{\heap{\e>0}{b\ge b_0}}\big(M(\e/b)b^{\frac{1}{1-\beta}}- M(\e)\big)>0.
\end{align*}
\end{lemma}

\begin{proof} Since $M$ is real-analytic and multiplicatively
periodic, it suffices to check that it satisfies the assumptions of Lemma~\ref{a1} with
$\gamma=1/(1-\beta)$. For fixed $n\in\N$ we define
\begin{align*}
B_n:=\big\{m/\mu^n: \P(Z_n=m)\neq 0\big\}
\end{align*}
and $B=\cup_{n\in\N}B_n$. We now show that $B$ is dense in $[1,\infty)$. Indeed, as $p_\mu\not=1$
there exists $\nu>\mu$ with $p_{\nu}\neq 0$. Denote $d=\nu-\mu$. We can prove by induction that
$$A_n:=\big\{m/\mu^n \colon \mu^n\le m\le \nu^n, m\equiv \mu^n(\text{mod } d)\big \}\subset B_n.$$
This is obvious  for $n=1$. Assuming that $\P(Z_{n-1}=r)\neq 0$ for all $r$ such that $\mu^{n-1}\le r\le \nu^{n-1}$,
$r\equiv \mu^{n-1}(\text{mod } d)$ we obtain that $\P(Z_n=m)\neq 0$ for all $m$ such that
there is $r$ satisfying the conditions above and such that $\mu r\le m\le \nu r$, $m\equiv \mu^n(\text{mod } d)$.
It is easy to see that this is equivalent to the condition $\mu^n\le m\le \nu^n, m\equiv \mu^n(\text{mod } d)$.
Hence $A_n\subset B_n$ and since $\cup_{n\in\N}A_n$ is dense in $[1,\infty)$ we obtain that
$\cup_{n\in\N}B_n$ is dense in $[1,\infty)$.
\smallskip

Let $b\in B$, that is, $b=m/\mu^n$
for some $m$ and $n$. We have
\begin{align*}
\P(Z_n=m\,|\,W<\e)\,\P(W<\e)=\P(W<\e, Z_n=m)=\P(W<\e\,|\,Z_n=m)\,\P(Z_n=m)
\end{align*}
and so
\begin{align*}
\liminf_{\e\downarrow 0}\frac{\P(W<\e)}{\P(W<\e\,|\,Z_n=m)}
=\liminf_{\e\downarrow 0}\frac{\P(Z_n=m)}{\P(Z_n=m\,|\,W<\e)}\ge \P(Z_n=m)>0.
\end{align*}
Hence
\begin{align}
\label{linf}
\liminf_{\e\downarrow 0} \big\{\log\P(W<\e)-\log\P(W<\e\,|\,Z_n=m)\big\}>-\infty.
\end{align}
Combining~\eqref{gqrks} with~\eqref{linf} we obtain
\begin{align*}
\liminf_{\e\downarrow 0}\left(M(\e/b)b^{\frac{1}{1-\beta}}-M(\e)\right)
=\liminf_{\e\downarrow 0}\left(M(\e \mu^n/m)(m/\mu^n)^{\frac{1}{1-\beta}}-M(\e)\right)
\ge 0,
\end{align*}
as required.
\end{proof}
\medskip

We now complete the proof of Theorem~1. Fix $n\in \N$ and use
Lemma~\ref{mmm} to find $c>0$\vspace{-1mm} (depending on $n$) such that
$M(\e/b)b^{\frac{1}{1-\beta}}- M(\e)\ge 2c$ for all $\e>0$ and $b\ge 1+\mu^{-n}$.
For all $m\ge \mu^n+1$ we have $m/\mu^n\ge 1+\mu^{-n}$ and hence \eqref{gqrks} implies
\begin{align*}
\log \P(W<\e)-\log \P(W<\e\,|\,Z_n=m)\ge c\e^{-\frac{\beta}{1-\beta}}.
\end{align*}
Therefore
\begin{align*}
\P\left(Z_n>\mu^n\,|\,W<\e\right) & 
 =\sum_{m=\mu^n+1}^{\infty}\frac{\P(W<\e\,|\,Z_n=m)}{\P(W<\e)}\,\P(Z_n=m)\\
& \le\sum_{m=\mu^n+1}^{\infty}\exp\left\{-c\e^{-\frac{\beta}{1-\beta}}\right\}\P(Z_n=m)
 \le \exp\left\{-c\e^{-\frac{\beta}{1-\beta}}\right\}\to 0
\end{align*}
as $\e\downarrow 0$, completing the proof of Theorem~1 in the case~$\mu>1$.

\begin{rem}\rm Lemma~\ref{mmm} can be seen as an illustration of the near-constancy phenomenon 
(see~\cite{biggins} and references therein), which consists in the fact that the function $M$
does not vary too much. Some numerical studies show that the variation of $M$ can be very small, and 
close theoretical bounds for $M$ are obtained for the case of an infinitely divisible distribution. 
No theoretical framework yet exists to describe the near-constancy of $M$ in the general case.
Lemma~\ref{mmm} implies that $$M(x)(y/x)^{\frac{1}{1-\beta}}-M(y)\ge 0 \quad
\mbox{for all $0<x\le y$,}$$ and so, with  $g(x)=x^{-\frac{1}{1-\beta}}$, the function $x\mapsto M(x)g(x)$ is decreasing. 
As the fluctuations of $M$ do not destroy the monotonicity of the decreasing function $g$ they cannot be too large. 
\end{rem}
\bigskip

{{\bf Acknowledgement:} The first author is supported by EPSRC grant EP/GO55068/1, and the second author is supported by an \emph{Advanced Research Fellowship} from EPSRC.}

\vspace{0.5cm}

\small 
Nathana\"el Berestycki: Statistical Laboratory, DPMMS, University of Cambridge. Wilberforce Rd., Cambridge CB3 0WB. United Kingdom.\\

Peter M\"orters: Department of Mathematical Sciences,
University of Bath.
Claverton Down,
Bath BA2 7AY.
United Kingdom.\\

Nadia Sidorova: Department of Mathematics, University College London.
Gower Street, London WC1E 6BT. United Kingdom.

\end{document}